\documentclass[a4paper,11pt]{amsart}
\usepackage{amsmath,amssymb,amsfonts,amsthm,enumerate}
\usepackage{epsfig}
\usepackage{mathrsfs}
\usepackage{todonotes}

\newtheorem{thm}{Theorem}[section]

\newtheorem{lma}{Lemma}[section]
\theoremstyle{remark}
\newtheorem{rmk}{Remark}[section]

\def\R{\mathbb{R}}
\def\E{\mathbb{E}}

\newcommand{\norm}[1]{\lVert#1\rVert}
\newcommand{\ind}[1]{\mathbf{1}_{#1}}
\newcommand{\abs}[1]{\left\vert#1\right\vert}
\newcommand{\ud}{\mathrm{d}}
\numberwithin{equation}{section}
\font\eka=cmex10
\usepackage{ae}
\def\ind{\mathrel{\hbox{\rlap{%
\hbox to 7.5pt{\hrulefill}}\raise6.6pt\hbox{\eka\char'167}}}}

\begin{document}
\title[Integral representation w.r.t. fBm]{Integral representation with adapted continuous integrand with respect to fractional Brownian motion}

\author[Shevchenko]{Georgiy Shevchenko}
\address{Department of Mechanics and Mathematics, Taras Shevchenko National University of Kyiv, Volodymirska 60, 01601 Kyiv, UKRAINE}

\author[Viitasaari]{Lauri Viitasaari}
\address{Department of Mathematics and System Analysis, Aalto University School of Science, Helsinki\\
P.O. Box 11100, FIN-00076 Aalto,  FINLAND} 

\begin{abstract}
We show that if a random variable is a final value of an adapted H\"older continuous process, then it can be represented as a stochastic integral with respect to fractional Brownian motion, and the integrand is an adapted process, continuous up to the final point.

\medskip


\smallskip

\end{abstract}

\maketitle

\section{Introduction}
In this paper we continue the research started in \cite{msv} and \cite{lauri}. Namely, we are interested in representations of the
\begin{equation}\label{representation}
\xi = \int_0^1 \psi(s) \ud B^H(s),
\end{equation}
where $B^H$ is a fractional Brownian motion with Hurst parameter $H>1/2$, and the integrand $\psi$ is adapted.  Such question is motivated by financial mathematics, where capitals of self-financing strategies are given by stochastic integrals with respect to asset pricing processes. 

In \cite{msv} it was proved that under some assumptions on the process $\xi$ the representation (\ref{representation}) exists. Moreover, it was also proved that under additional assumption that $\psi(s)$ is continuous, the assumptions are also necessary. In this paper we extend the results of \cite{msv}, showing the existence of the representation \eqref{representation} with $\psi\in C[0,1)$ under the same assumptions as in \cite{msv}.

The paper is organized as follows. Section~\ref{sec:prelim} provides necessary information on the fractional Brownian motion and the definition of stochastic integral we use. Section~\ref{sec:main} contains the main result on the existence of a representation with a continuous integrand. Section~\ref{sec:blahblahblah} concludes the article with a discussion of our findings and their  comparison with existing results.

\section{Preliminaries}\label{sec:prelim}
Let $(\Omega,\mathcal{F},\mathbb{F} = \{\mathcal{F}_t,t\in [0,1]\}, P)$ be a stochastic basis, and $\{B^H(t),t\in [0,1]\}$ be a fractional Brownian motion (fBm) on this stochastic basis, i.e.\ an $\mathbb{F}$-adapted centered Gaussian process with the covariance function
$$
\E {B^H(t) B^H(s)} = \frac12 \left( s^{2H} + t^{2H} - | s-t| ^{2H} \right), \ s,t \in [0,1].
$$
It is well known that the fBm $B^H$ has a continuous modification, moreover, this modification is H\"older continuous of any order $\gamma\in(0,H)$. In what follows we will assume that $B^H$ is continuous.

We will understand the integral with respect to the fBM $B^H$ as the generalized fractional Lebesgue--Stieltjes integral. Below we give only basic information on the integral, the details can be found in \cite{zahle}. 

Let  $f,g\colon[a,b]\to \R$, $\beta\in (0,1)$. Define the forward and backward fractional Riemann--Liouville derivatives
\begin{gather*}
\big(D_{a+}^{\beta}f\big)(x)=\frac{1}{\Gamma(1-\beta)}\bigg(\frac{f(x)}{(x-a)^\beta}+\beta
\int_{a}^x\frac{f(x)-f(u)}{(x-u)^{1+\beta}}du\bigg),\\
\big(D_{b-}^{1-\beta}g\big)(x)=\frac{e^{i\pi
\beta}}{\Gamma(\beta)}\bigg(\frac{g(x)}{(b-x)^{1-\beta}}+(1-\beta)
\int_{x}^b\frac{g(x)-g(u)}{(u-x)^{2-\beta}}du\bigg),
\end{gather*}
where $x\in(a,b)$.

The generalized Lebesgue--Stieltjes
integral 
is defined as
\begin{equation*}\int_a^bf(x)\ud g(x)=e^{-i\pi\beta}\int_a^b\big(D_{a+}^{\beta}f\big)(x)\big(D_{b-}^{1-\beta}g_{b-}\big)(x)\ud x
\end{equation*}
provided the integral in the right-hand side exists. Thanks to the H\"older continuity of $B^H$, it can be easily shown that for any $\beta\in(1-H,1)$,  
\begin{equation}\label{finitederiv}
\sup_{0\le u<s\le 1} \abs{D_{u-}^{1-\beta}B^H_{s-}}<\infty.
\end{equation}
 Therefore, for
$f$ with $D_{a+}^\beta f\in L_1[a,b]$ we can define the integral with respect to
$B^H$ as
\begin{equation*}
\int_a^bf(s)\ud B^H(s)=e^{-i\pi\beta}\int_a^b\big(D_{a+}^{\beta}f\big)(x)\big(D_{b-}^{1-\beta}B^H_{b-}\big)(x)\ud x.
\end{equation*}
Now fix some $\beta \in(1-H,1/2)$ and introduce the following norm:
\begin{gather*}
\norm{f}_{\beta,[a,b]} = \int_a^b \left(\frac{\abs{f(s)}}{(s-a)^\beta} + \int_a^s \frac{\abs{f(s)-f(z)}}{(s-z)^{1+\beta}}dz\right)ds
\end{gather*}
and abbreviate $\norm{f}_{\beta,t} = \norm{f}_{\beta,[0,t]}$.
In view of \eqref{finitederiv}, the integral $\int_0^t f_s \ud B^H_s$ is well defined once $\norm{f}_{\beta,t}<\infty$. Hence, the class of admissible integrands for us will consist of the processes satisfying this condition. We remark that in particular the integral is well defined once $f$ is H\"older continuous of some order $\nu>1-H$; in this case it is equal to the limit of integral sums, i.e.\ to the so called Young integral.

For our main result, we will need the following change of variable formula for fBm.
\begin{thm}[\cite{amv}]\label{ito}
Let $f:\R\to \R$ be a function of locally bounded variation, $F(x) = \int_0^x f(y) dy$. Then for any $\beta\in(1-H,1/2)$
\ $\norm{f(B^H(\cdot))}_{\beta,1}<\infty$ a.s. and
$$
F(B^H(t)) = \int_0^t f(B^H(s)) \ud B^H(s),\quad t\in[0,1].
$$
\end{thm}

Throughout the article, we will use the symbol $C$ to denote a generic constant, whose value may change from one line to another. If the constant in question depends on some parameters, we will write them in a subscript. 

\section{Main results}
\label{sec:main}

In this section we establish, under certain conditions on $\xi$, the existence of representation \eqref{representation} with an integrand $\psi\in C[0,1)$. The techniques are similar to those used in \cite{msv}, so we will omit some minor details.

As in \cite{msv}, we start with an auxiliary lemma. 
\begin{lma}
\label{lemma}
There exists a $\mathbb{F}$-adapted continuous process $\phi$ on $[0,1]$ such that $\phi(0) = 0$, the integral
\begin{equation*}
\int_0^t \phi(s)\ud B^H(s)
\end{equation*}
exists for every $t<1$ and 
\begin{equation}
\label{rep:aux-lemma}
\lim_{t\to 1-} \int_0^t \phi(s)\ud B^H(s) = +\infty
\end{equation}
almost surely.
\end{lma}
\begin{proof}
Fix numbers $\gamma\in\left(1,\frac{1}{H}\right)$, $\eta\in\left(0,\frac{1}{\gamma H}-1\right)$ and $\mu> \frac1{H(1+\eta)}$. 
Set $t_0 = 0$ and $t_n = \sum_{k=1}^{n}(\Delta_{k} + \widetilde\Delta_{k})$, where $\Delta_k = K k^{-\gamma}$, $\widetilde\Delta_k = K k^{-\mu}$, $K = \big(\sum_{k=1}^\infty (k^{-\gamma}+k^{-\mu})\big)^{-1}$. Also set $t_n' = t_{n-1} + \Delta_n$, $n\ge 1$. Clearly, $t_{n-1}<t_n' < t_n$, $n\ge 1$, and $t_n\to 1$, $n\to\infty$. 

Define the sequence of functions $g_n = \sqrt{x^2 + n^{-2}} - n^{-1}$, $n\ge 1$. Then $g_n(x) \uparrow |x|$, $n\to\infty$. Let also $f_n = (1+\eta) g_n(x)^\eta \frac{x}{\sqrt{x^2+n^{-2}}}$ so that $g_n(x)^{1+\eta} = \int_0^x f_n(z)\ud z$. 

Now we proceed to the construction of the process $\phi$. It is similar to the one used in \cite[Lemma 3.1]{msv}, but we need to make some adjustments to make the integrand continuous. 

For any $n\ge 1$ set
\begin{equation*}
\tau_n = \min\left\{t\geq t_{n-1} : |B^H(t) - B^H({t_{n-1}})| \geq n^{-{1}/({1+\eta})}\right\} \wedge t_n'.
\end{equation*}
Define  
\begin{equation*}
\phi(s) = f_{n}(B^H(s) - B^H(t_{n-1}))\textbf{1}_{[t_{n-1},\tau_n)}(s)
\end{equation*}
for $s\in[t_{n-1},\tau_n]$
and 
\begin{equation*}
\phi(s) = \phi(\tau_n)\frac{\tau_n + \widetilde \Delta_n-s}{\widetilde{\Delta}_n}\textbf{1}_{(\tau_n,\tau_n+\widetilde\Delta_n]}(s)
\end{equation*}
for $s\in(\tau_n,t_{n}]$. Evidently, $\phi(t_{n-1}) = \phi(t_n) = 0$ and $\phi\in C[t_{n-1},t_n]$ Therefore, $\psi \in C[0,1)$. The existence of the integral $\int_0^t \psi(s) \ud B^H(s)$ for any $t>1$ is clear. Indeed, for each $n\ge1$, on the interval $[t_n,\tau_n]$, $\phi$ is a smooth transformation of $B^H$, therefore, H\"older continuous of any order $\nu<H$, and on the interval $[\tau_n, t_{n}]$, it is piecewise linear.

In order to complete the proof, we need to show that (\ref{rep:aux-lemma}) holds. First by Theorem~\ref{ito}, we obtain  for any $n\ge 1$
\begin{equation*}
\int_{t_{n-1}}^{\tau_n} \phi(s)\ud B^H(s) = g_{n}(B^H({\tau_n})-B^H(t_{n-1}) )^{1+\eta}.
\end{equation*}
Hence, by additivity,
\begin{equation*}
\begin{split}
\int_0^{t_n} \phi(s)\ud B^H_s = \sum_{k=1}^n g_{k}(B^H_{\tau_k}-B^H_{t_{k-1}})^{1+\eta}+ \sum_{k=1}^{n} \int_{\tau_k}^{\tau_k + \widetilde\Delta_k} 
\phi (s)\ud B^H(s).
\end{split}
\end{equation*}
Now we will show that the first sum diverges to $+\infty$, while the second one converges.

Thanks to the Jensen's inequality, $(a+b)^{1+\eta} \leq 2^{\eta}(a^{1+\eta}+b^{1+\eta})$ for $a,b>0$, 
which implies
$$
g_{k}(B^H_{\tau_k}-B^H_{t_{k-1}})^{1+\eta} \ge 2^{-\eta}|B^H_{\tau_k}-B^H_{t_{k-1}}|^{1+\eta} - k^{-1-\eta}.
$$
From the proof of \cite[Lemma 3.1]{msv} it follows that
$$
\sum_{k=1}^\infty |B^H_{\tau_k}-B^H_{t_{k-1}}|^{1+\eta} = +\infty
$$
almost surely. Therefore, since $\sum_{k=1}^{\infty}k^{-1-\eta}<\infty$, we get $$ \sum_{k=1}^n g_{k}(B^H_{\tau_k}-B^H_{t_{k-1}})^{1+\eta} = +\infty$$ 
almost surely.

Further, 
$$
|f_{k}(x)|\leq (1+\eta)|x|^{\eta}
$$
and hence, by the definition of $\tau_k$, 
\begin{equation}\label{phibound}
|\phi(\tau_k)| \leq (1+\eta)n^{-\eta/(1+\eta)}.
\end{equation}
Consequently, we can estimate
\begin{align*}
\left|\int_{\tau_k}^{\tau_k + \widetilde\Delta_k} 
\phi (s)\ud B^H(s)\right|& = \left|\phi(\tau_k)\widetilde{\Delta}_k^{-1}
\int_{\tau_k}^{\tau_k + \widetilde\Delta_k} ({\tau_k + \widetilde \Delta_k-t}) \ud B^H(s)\right|\\
&\le C k^{-\eta/(1+\eta)}\widetilde{\Delta}_k^{-1}\left|\int_{\tau_k}^{\tau_k + \widetilde\Delta_k} \left({B^H(s) - B^H(\tau_k)}\right) \ud s\right|\\
&\le C k^{-\eta/(1+\eta)} \sup_{s\in[\tau_k,\tau_k+\widetilde{\Delta}_k]}\left|{B^H(s) - B^H(\tau_k)}\right|.
\end{align*}
By the H\"older properties of $B^H$, we have for any $\epsilon>0$
\begin{align*}
\left|\int_{\tau_k}^{\tau_k + \widetilde\Delta_k} 
\phi (s)\ud B^H(s)\right|& \le C_\epsilon(\omega) k^{-\eta/(1+\eta)} \widetilde{\Delta}_k^{H-\epsilon}\le C_\epsilon(\omega) k^{-\eta/(1+\eta)-\mu({H-\epsilon})}.
\end{align*}
Since $\mu> \frac1{H(1+\eta)}$, we have $\eta/(1+\eta)+\mu({H-\epsilon})> 1$ for some $\epsilon>0$, which implies that 
$$
\sum_{k=1}^\infty\left|\int_{\tau_k}^{\tau_k + \widetilde\Delta_k} 
\phi (s)\ud B^H(s)\right|<+\infty.
$$
The proof is finished by writing 
$$
\int_0^t \phi(s)\ud B^H(s) = \int_0^{t_n} \phi(s)\ud B^H(s) + \int_{t_n}^{t} \phi(s)\ud B^H(s)
$$
for $t\in[t_n,t_{n+1})$ and observing that the latter integral is bounded in $n$ and $t$  thanks to the above estimations.
\end{proof}
\begin{rmk}\label{phiremark}
It is evident from the proof that $\abs{\phi(t)}\le 1+\eta$ for any $t\in[0,1]$. Indeed, for $t = \tau_n$ this is true by \eqref{phibound}, for $t\in[t_{n-1},\tau_n]$, by the definition of $\tau_n$, and for $t\in[\tau_n,t_n]$, by the definition of $\phi$.
\end{rmk}

Let us now turn to the main theorem which provides a representation of a random variable as a stochastic integral of a continuous process.
\begin{thm}
\label{thm:H_rv-rep}
Assume there exists an $\mathbb F$-adapted process $\{z(t), t\ge 1\}$ having  H\"{o}lder continuous  paths of order
$a>0$ and such that $z(1) = \xi$. Then there exists an $\mathbb{F}$-adapted process $\{\psi(t), t\in[0,1]\}$  such that  $\psi \in C[0,1)$ a.s.\ and
\begin{equation}\label{repres}
\int_0^1 \psi(s)\ud B^H(s) = \xi
\end{equation}
almost surely.
\end{thm}
\begin{proof}
We modify the proof \cite[Theorem 3.3]{msv} to show the existence of a continuous integrand. To facilitate reading, we divide the proof into four logical steps.

\textbf{Step 1. Choice of parameters and introduction of notation.}
Without loss of generality, assume that $a<H$. Firstly, choose some $\gamma>(1-\beta-H+a)^{-1}>1$. Next, it is straightforward to check the inequalities $\gamma(H-a)< \frac{\gamma (H-\beta a) - H}{\beta + H}<\gamma(1-\beta)-1$ (both are reduced to the inequality $\gamma(1-\beta-H+a)>1$). Therefore, one can choose $\kappa$ such that
$$
\gamma(H-a)< \kappa < \frac{\gamma (H-\beta a) - H}{\beta + H}< \gamma(1-\beta)-1.
$$
The middle inequality is equivalent to $\gamma a + \kappa < (\gamma-\kappa - 1) \frac{H}{\beta}$, hence, it is possible to choose $\nu$ such that 
$$
\gamma a + \kappa < \nu < (\gamma-\kappa - 1) \frac{H}{\beta}.
$$
Finally, since $\kappa<\gamma(H-\beta a) - H$ and $a<H$, we get
$$
\gamma a + \kappa < \gamma(H + (1-\beta)a) - H < \gamma(H + (1-\beta)H) - H = H(\gamma(2-\beta)-1).
$$
Therefore, we can choose $\mu$ such that 
$$
(\gamma a + \kappa)/H < \mu < \gamma(2-\beta) - 1.
$$
Observe also that from  $\kappa> \gamma (H-a)$ it follows that $\mu>\gamma$.

Define  $t_0 = 0$ and $t_n = \sum_{k=1}^{n}\Delta_{k}$, where $\Delta_k = K k^{-\gamma}$, $K = \big(\sum_{k=1}^\infty (k^{-\gamma})\big)^{-1}$; also set $t_n' = t_{n-1} + \Delta_n/2$, $n\ge 1$.  Denote $g_{n}(x) = \sqrt{x^2+n^{-2\nu}}-n^{-\nu}$, $\widetilde{\Delta}_n = Kn^{-\mu}/2$.

\textbf{Step 2. Construction}. The construction is by induction. 
We start by setting $\psi(t) = 0$ on the interval $[t_0,t_1]$. We also set $\tau_1 = t_1$. 

Denote $y(t) = \int_0^t \psi(s) dB^H(s)$, $\xi_n = z({t_{n-1}})$, $n\ge 1$.

Now let $n\ge 2$. 
We want to define the process $\psi$ on the  interval $[t_{n-1}, t_{n}]$ such that
\begin{enumerate}[(P1)]
\item $\psi(t_{n-1}) = \psi(t_{n}) = 0$;
\item $y(\tau_{n}) = \xi_{n}$ for some $\tau_{n}\in [t_{n-1},t'_{n}]$;
\item $\psi$ is linear on $[\tau_{n},\tau_{n}+\widetilde\Delta_n]$ and zero afterwards, that is,
\begin{equation}\label{psilinear}
\psi(t) = \psi(\tau_n)\frac{\tau_n + \widetilde \Delta_n-t}{\widetilde{\Delta}_n}\textbf{1}_{(\tau_n,\tau_n+\widetilde\Delta_n]}(t),\quad t\in [\tau_n, t_{n}].
\end{equation}
\end{enumerate}
We remark at once that the properties (P1) and (P3) will hold for all $n$ by construction, but (P2) will take place only starting from some (random) $n$.

Case A) $y({\tau_{n-1}})=\xi_{n-1}$. Define 
\begin{equation*}
\tau_{n} = \inf\left\{t\geq t_{n-1} : n^{\kappa}g_{n}(B^H(t)- B^H(t_{n-1}))= \left|\Lambda_n\right|\right\} \wedge t_{n}',
\end{equation*}
where 
$
\Lambda_n = \xi_n - y(t_{n-1}) = \xi_n - \xi_{n-1} -\int_{\tau_{n-1}}^{t_{n-1}}\psi(s)\ud B^H(s).
$
Put $$\psi(t) = n^\kappa g'_{n} (B_t^H - B_{t_{n-1}}^{H})\operatorname{sign} \Lambda_n,\quad t\in [t_{n-1}, \tau_{n}]$$
and define it by \eqref{psilinear} on $[\tau_{n},t_{n}]$.

From Theorem~\ref{ito} it follows that 
$$
y(t) = y(t_{n-1}) + n^\kappa g_{n}(B_t^H - B_{t_{n-1}}^{H})\operatorname{sign} \Lambda_n,\quad t\in [t_{n-1}, \tau_{n}];
$$
in particular,  $y(\tau_{n}) = \xi_n$ provided that $\tau_{n}<t'_{n}$. 

Case B)  $y({\tau_{n-1}}) \neq \xi_{n-1}$. It follows from Lemma~\ref{lemma} that there exists an adapted continuous process $\{\phi_n(t),t\in[t_{n-1},t_n']\}$ such that $v_n(t):=\int_{t_{n}}^t \phi_n(s) dB^H(s)\to \infty$, $t\to t'_n-$. Therefore we can define the stopping time $\tau_{n} = \inf \{t\in [t_{n-1},t_{n}'): v(t) = \abs{ \xi_n-y(\tau_{n-1})}\}$. Then we put $\psi (t) = \phi_n(t)\operatorname{sign}(\xi_n - y(\tau_{n-1}))$, $t\in [t_{n-1},\tau_n]$, and use \eqref{psilinear} on $[\tau_{n},t_{n}]$. Clearly, $y(\tau_n) = \xi_n$.

\textbf{Step 3. (P2) takes place for all $\pmb{n}$ large enough.}
We need to show that we have Case B) only finite number of times. For this we need to show that, almost surely, only finite numbers of events
\begin{equation*}
A_n = \Big\{\sup_{t\in[t_{n-1},t_n')} n^{\kappa}g_{n}(B^H_t- B^H_{t_{n-1}})\leq |\Lambda_n|\Big\}
\end{equation*}
happens. This is done similarly to the proof of \cite[Theorem 3.3]{msv}, but we need to take care of the extra term $\int_{\tau_{n-1}}^{t_{n-1}}\psi(s)\ud B^H(s)$ and of the function $g_n$. As in the proof of Lemma~\ref{lemma}, we have for any $\epsilon>0$
\begin{align*}
\abs{\int_{\tau_{n-1}}^{t_{n-1}}\psi(s)\ud B^H(s)}\le C_\epsilon (\omega)\abs{\phi(\tau_{n-1})} \widetilde{\Delta}_{n-1}^{H-\epsilon}\le C_\epsilon (\omega)\abs{\phi(\tau_{n-1})} n^{-\mu(H-\epsilon)}.
\end{align*}
Now we have $\abs{\psi(\tau_{n-1})}\le n^\kappa \sup_{x\in\R} \abs{g'_n(x)}\le n^\kappa$ in Case A) and $\abs{\psi(\tau_{n-1})}\le C$ in Case B) (see Remark~\ref{phiremark}). Therefore, 
$$
\abs{\int_{\tau_{n-1}}^{t_{n-1}}\psi(s)\ud B^H(s)}\le C_\epsilon(\omega) n^{\kappa-\mu(H-\epsilon)}.
$$
Thanks to the choice of $\mu$, $\kappa - \mu(H-\epsilon)< - \gamma a$ for $\epsilon$ small enough, so we get
\begin{equation}\label{kusok}
\abs{\int_{\tau_{n-1}}^{t_{n-1}}\psi(s)\ud B^H(s)}\le C_\epsilon(\omega) \Delta_n^a.
\end{equation}
Further, $g_n(x)\ge \abs{x}-n^{-\nu}$. Therefore, the event $A_n$ implies that 
$$
\sup_{t\in[t_{n-1},t_n')} n^{\kappa}\abs{B^H_t- B^H_{t_{n-1}}}\leq |\Lambda_n| + n^{\kappa - \nu}.
$$
By the assumption on $\nu$, we get that $n^{\kappa-\nu}= o(\Delta_n^{a})$, $n\to\infty$. In view of \eqref{kusok}, we get that $A_n$ implies 
$$
\sup_{t\in[t_{n-1},t_n')} n^{\kappa}\abs{B^H_t- B^H_{t_{n-1}}}\leq C(\omega)\Delta_n^a
$$
for all $n$ large enough, and the further proof goes exactly the same way as the proof of Step 2 in \cite[Theorem 3.3]{msv}. 

\textbf{Step 4. The construction provides the desired representation.} We need to show that 
$||\psi||_{1,\beta}<\infty$ almost surely and $\int_0^1\psi(s) \ud B^H(s) = \xi$. We remark that there is a small gap in the proof of \cite[Theorem 3.3]{msv}: there it is proved in fact that $\|\psi\|_{1,\beta}<\infty$ and $y(t)\to \xi$, $t\to 1-$, a.s. This, however, is not enough to guarantee \eqref{repres}, since  $y$ might be discontinuous at $1$. 

In Step 3, we proved that $y(\tau_n) = \xi_n$ for all $n$ large enough, say, for $n\ge N(\omega)$. Since $\xi_n \to \xi$, $n\to\infty$, a.s., it is enough to show that $\int_{\tau_n}^{1} \psi(s)\ud B^H(s) \to 0$, $n\to\infty$, a.s., which, in turn, would follow from $\norm{\psi}_{\beta, [\tau_n,1]}\to 0$, $n\to\infty$, a.s.

Assume further that $n\ge N(\omega)$.
Write 
$$
\norm{\psi}_{\beta, [\tau_n,1]} = I_1 + I_2, 
$$
where 
$$
I_1 =  \int_{\tau_n}^1\frac{\abs{\psi(t)}}{(t-\tau_n)^\beta} \ud s,\quad I_2  = \int_{\tau_n}^1\int_{\tau_n}^{t}\frac{\abs{\psi(t)-\psi(s)}}{(t-s)^{\beta+1}}\ud s\,\ud t.
$$
Estimate
\begin{align*}
I_1 & = \int_{\tau_n}^{t_n}\frac{\abs{\psi(t)}}{(t-\tau_n)^\beta}\ud t +  \sum_{k=n}^{\infty} \int_{t_k}^{t_{k+1}}\frac{\abs{\psi(t)}}{(t-\tau_n)^{\beta}}\ud t\\
& \le C n^\kappa {\Delta}_n^{1-\beta} + C\sum_{k=n}^{\infty} \frac{k^{\kappa}\Delta_k}{(t_k - \tau_{n-1})^{\beta}} \le Cn^{\kappa - \gamma(1-\beta)} + C {\Delta}_n^{-\beta} \sum_{k=n}^{\infty} k^{\kappa - \gamma}\\
&\le C n^{\kappa- \gamma(1-\beta)+1}\to 0,\quad n\to\infty.
\end{align*}
The last inequality here  is due to the choice of $\kappa$. Further,
\begin{align*}
I_2 & = 
\sum_{k=n}^\infty \int_{t_k}^{\tau_{k+1}}  \int_{\tau_n}^{t} \psi(t,s)\ud s \,\ud t 
+ \sum_{k=n}^\infty \int_{\tau_{k}}^{t_{k}}  \int_{\tau_n}^{t} \psi(t,s)\ud s \,\ud t\\
&  = \sum_{k=n}^\infty \int_{t_k}^{\tau_{k+1}}  \int_{\tau_n}^{t_k} \psi(t,s)\ud s \,\ud t 
+ \sum_{k=n}^\infty \int_{\tau_{k}}^{t_{k}}  \int_{\tau_n}^{\tau_k} \psi(t,s)\ud s \,\ud t\\
& + \sum_{k=n}^\infty \int_{t_k}^{\tau_{k+1}}  \int_{t_k}^{t} \psi(t,s)\ud s \,\ud t 
+ \sum_{k=n}^\infty \int_{\tau_{k}}^{t_{k}}  \int_{\tau_k}^{t} \psi(t,s)\ud s \,\ud t,
\end{align*}
where $\psi(t,s) = \abs{\psi(t)-\psi(s)}(t-s)^{-\beta-1}$. 
Next we estimate the terms one by one. First,
\begin{align*}
&\sum_{k=n}^{\infty}\int_{t_k}^{\tau_{k+1}}  \int_{\tau_n}^{t_k} \psi(t,s)\ud s \,\ud t\le 2\sum_{k=n}^{\infty} k^{\kappa}\int_{t_k}^{\tau_{k+1}}  \int_{\tau_n}^{t_k}\frac{\ud s\, \ud t}{(t-s)^{\beta + 1}}\\& \le C\sum_{k=n}^{\infty} k^\kappa\int_{t_k}^{\tau_{k+1}}\frac{\ud t}{(t-\tau_n)^{\beta}}
 \le C \sum_{k=n}^{\infty}k^{\kappa} {\Delta}_n^{1-\beta}\le Cn^{\kappa - \gamma(1-\beta)+1}\to 0,\quad n\to\infty,
\end{align*}
the last is due to the choice of $\kappa$. 
Similarly, 
$$
\sum_{k=n}^\infty \int_{\tau_{k}}^{t_{k}}  \int_{\tau_n}^{\tau_k} \psi(t,s)\ud s \,\ud t\to 0, \quad n\to\infty.
$$
Further, denote $\delta_k = k^{-\nu/H}$ and  write 
$$
\int_{t_k}^{\tau_{k+1}}  \int_{t_k}^{t} \psi(t,s)\ud s \,\ud t  = \int_{t_k}^{\tau_{k+1}}\left(\int_{t_k}^{t-\delta_k}+\int_{t-\delta_k}^{t}\right)\psi(t,s)\ud s \,\ud t.
$$
Now 
\begin{align*}
&\sum_{k=n}^\infty\int_{t_k}^{\tau_{k+1}}\int_{t_k}^{t-\delta_k}\psi(t,s)\ud s \,\ud t\le 2\sum_{k=n}^\infty k^{\kappa}\int_{t_k}^{\tau_{k+1}}\int_{t_k}^{t-\delta_k}(t-s)^{-\beta-1}\ud s\, \ud t\\
&\le C\sum_{k=n}^\infty k^{\kappa}\int_{t_k}^{\tau_{k+1}}\delta_k^{-\beta}\ud t\le C\sum_{k=n}^{\infty}  k^{\kappa-\gamma+\beta \nu/H}\le Cn^{\kappa-\gamma+\beta \nu/H+1}\to 0, \quad n\to\infty.
\end{align*}
To estimate the second part, observe that $\abs{g_k''(x)}\le k^{\nu}$. Therefore, for $\epsilon>0$ small enough,
\begin{align*}
&\sum_{k=n}^{\infty}\int_{t_k}^{\tau_{k+1}}  \int_{t-\delta_k}^{t} \frac{\abs{\psi(t)-\psi(s)}}{(t-s)^{\beta+1}}\ud s \,\ud t\le \sum_{k=n}^{\infty}
k^{\kappa+\nu}\int_{t_k}^{\tau_{k+1}}  \int_{t-\delta_k}^{t}\frac{\abs{B^H(t)-B^H(s)}}{(t-s)^{\beta+1}}\ud s \,\ud t\\
& \le C_\epsilon(\omega)\sum_{k=n}^{\infty}k^{\kappa+\nu}\int_{t_k}^{\tau_{k+1}}  \int_{t-\delta_k}^{t}(t-s)^{H-\epsilon-\beta-1}\ud s \,\ud t
\le C_\epsilon(\omega) \sum_{k=n}^{\infty} k^{\kappa+\nu} \Delta^{\vphantom{H-\beta}}_k \delta_k^{H-\epsilon-\beta}\\&\le C_\epsilon(\omega) \sum_{k=n}^{\infty} k^{\kappa+\nu-\gamma - \nu (H-\epsilon-\beta)/H} \le C_\epsilon(\omega )n^{\kappa-\gamma+\beta \nu/H +\epsilon/H +1}\to 0, \quad n\to\infty.
\end{align*}
Finally, observing that $\abs{\psi(t)-\psi(s)}\le \Delta_k^{-1}\abs{t-s}$ for $t,s\in[\tau_k,t_k]$ and $\mu - \gamma(2-\beta)<-1$, we get
\begin{align*}
&\sum_{k=n}^\infty \int_{\tau_{k}}^{t_{k}}  \int_{\tau_k}^{t} \psi(t,s)\ud s \,\ud t \le \sum_{k=n}^\infty\widetilde{\Delta}_k^{-1}\int_{\tau_{k}}^{t_{k}}  \int_{\tau_k}^{t} (t-s)^{-\beta}\ud s\, \ud t\\
& \le C\sum_{k=n}^{\infty} \widetilde{\Delta}_k^{-1} \Delta_k^{2-\beta} \le C\sum_{k=n}^{\infty} k^{\mu - \gamma(2-\beta)}\le C n^{\mu - \gamma(2-\beta)+1}\to 0,\quad n\to\infty.
\end{align*}
The proof is now complete.

\end{proof}

\section{Discussion} \label{sec:blahblahblah}

In this section we give previous related results and compare them to ours. 

The first stochastic representation result is the classical It\^o theorem, which says the following. Let the filtration $\mathbb F^W = \{\mathcal F^W_t, t\in[0,1]\}$ be generated by a standard Wiener process. Then any centered quadratic integrable $\mathcal F^W_1$-measurable random variable $\xi$ can be represented in a form
\begin{equation}\label{wienerrepr}
\xi = \int_{0}^{1} \psi(s) \ud W(s),
\end{equation}
where $\psi$ is an $\mathbb F^W$-adapted process such that $\int_0^1 \E \psi(s)^2 \ud s<+\infty$. However, it is well known that the It\^o integral can be defined once $\int_0^1 \psi(s)^2 \ud s<\infty$ a.s. For such extended definition, Dudley in \cite{dudley} showed that \textit{every} $\mathcal F_1^W$-measurable random variable is representable in the form \eqref{wienerrepr}. 

For fractional Brownian motion, first representation results of such type were obtained in \cite{msv}. Like in this article, the assumption of the main result in \cite{msv} is that there exists a H\"older continuous adapted process $\{z(t),t\in[0,1]\}$ such that $z(1) = \xi$ (let us call it the H-assumption). The authors showed in \cite{msv} that the H-assumption implies the existence of the representation \eqref{repres}, and, moreover,  that the existence of the representation \eqref{repres} with $\psi\in C[0,1]$ implies the H-assumption. In this paper we were able to show that the H-assumption implies the existence of the representation \eqref{repres} with $\psi\in C[0,1)$, so now we are much closer to some criterion.

Some remarks about generalizations. Firstly, in \cite{msv} for \textit{every} random variable the existence of an improper representation was shown, namely, that there exists an adapted process $\psi$ such that the integral $y(t) = \int_0^t \psi(s)\ud B^H(s)$ exists for every $t<1$ and $\lim_{t\to 1-} y(t) = \xi$. It is not difficult to prove a similar result with $\psi\in C[0,1)$ by using our auxiliary construction of Lemma~\ref{lemma}; however, as in \cite{msv}, we need an additional assumption that the filtration $\mathbb F$ is continuous at $1$, i.e. $\mathcal F_1 = \sigma\left(\bigcup_{t<1} \mathcal F_t\right)$. Secondly, in \cite{lauri} the results of \cite{msv} are generalized to a wider class of Gaussian processes. Similarly, it is possible to prove Theorem~\ref{thm:H_rv-rep} for a bigger class of processes (not even necessarily Gaussian): in fact, only properties needed are H\"older continuity and some small ball estimates.

Finally, we remark that there is an alternative definition of a stochastic integral with respect to fractional Brownian motion. It appears under different names in the scientific literature: the divergence integral, the Skorokhod integral, or the white noise integral. For this definition, in \cite{hu-oksendal}, see also \cite{elliott-hoek}, an analogue of Clark--Ocone formula was proved, which gives a representation \eqref{repres} with the divergence integral.  We also  mention paper  \cite{bender}, where the uniqueness of an adapted integrand in such representation is established. This is in sharp contrast with our findings for pathwise integral, where the representation is very far from being unique.

\vskip0.5cm

\bibliographystyle{plain}      
\bibliography{cont_integrand}

\end{document}